\documentclass[12pt]{article} 
\usepackage{graphicx}
\usepackage{amsmath,amssymb,textcomp,graphicx,amsthm} 
\usepackage{epsfig}
\usepackage{color}
\usepackage[english]{babel} 
\usepackage[alphabetic]{amsrefs}
 \usepackage{verbatim}
\usepackage[latin1]{inputenc}  

\numberwithin{equation}{section}

\newcommand{\e}{\varepsilon}

\newcommand{\R}{\mathbb{R}}
\newcommand{\N}{\mathbb{N}}

\newtheorem{de}{Definition}
\newtheorem{lem}{Lemma}
\newtheorem{prop}{Proposition}
\newtheorem{thm}{Theorem}
\newtheorem{cor}{Corollary}
\newtheorem{rem}{Remark}

\newcommand{\dd}{\partial}

\newcommand{\supp}{\operatorname{supp}}

\newcommand{\lap}{\Delta}

\renewcommand{\div}{\operatorname{div}}

\newcommand{\ont}{\int_0^T\!\!\!\int_\Omega}

\renewcommand{\O}{\mathcal{O}}
\title{On a comparison principle for Trudinger's equation}

\author{Erik Lindgren and Peter Lindqvist}
\begin{document}

\maketitle

\begin{abstract}
\noindent We study the comparison principle for non-negative solutions of the equation 
$$
\frac{\dd\,(|v|^{p-2}v)}{\dd t}\,=\, \div (|\nabla v|^{p-2}\nabla v), \quad 1<p<\infty.
$$
This equation is related to extremals of Poincar\'e inequalities in Sobolev spaces. We apply our result to obtain pointwise control of the large time behavior of solutions.
\end{abstract}

\bigskip 

\noindent {\small \textsf{AMS Classification 2010}: 35K65, 35K55, 35B40, 35A02, 35D30, 35D40} 


\bigskip

\section{Introduction}
 Among the so-called doubly non-linear evolutionary equations, Trudinger's equation 
 \begin{equation}
\label{eq:maineq}
\frac{\dd\,(|v|^{p-2}v)}{\dd t}\,=\, \div (|\nabla v|^{p-2}\nabla v), \quad 1<p<\infty,     
\end{equation}
is distinguished. We shall study its solutions in $\Omega_T = \Omega \times (0,T),$ where $\Omega$ is a domain in $\R^n.$ The equation was originally considered by Trudinger in \cite{Tru}, where a Harnack inequality was studied for a wider class of evolutionary equations. He pointed out that no ``intrinsic scaling'' is needed for (\ref{eq:maineq}).
The equation has two special features:  it is \emph{homogeneous} and  it is \emph{not translation invariant}, except for $p=2$ when it reduces to the Heat Equation. The first property is, of course, an advantage.

While some progress has been made regarding continuity and regularity properties, the question about \emph{uniqueness} for the Dirichlet boundary value problem seems to be unsettled (under natural assumptions). Sign-changing solutions certainly come with many challenges.



We shall prove a comparison principle in Theorem \ref{thm:comp} for positive weak supersolutions/subsolutions  belonging, by definition, to the Sobolev space $L^p(0,T;W^{1,p}(\Omega))$. One of the functions is required to be bounded away from zero. Nevertheless, we can allow that one of the functions has zero lateral boundary values (Corollary \ref{classic}), which is relevant for a related  eigenvalue problem. Furthermore,  for Perron's method, a proper comparison principle is a \emph{sine qua non}. Our result also implies that for $p\geq 2$, all non-negative continuous weak solutions are viscosity solutions in the sense of Crandall,  Evans, and Lions in \cite{useless}. This is Theorem \ref{visco}. 

The equation
has an interesting connection to extremals of Poincar\'e inequalities in the Sobolev space $W_0^{1,p}(\Omega)$. These are the minimizers of the Rayleigh quotient
\begin{equation}\label{rayleigh}
\lambda_{p} = \inf_{u\in W_0^{1,p}(\Omega)\setminus \{0\}}\frac{\displaystyle\int_\Omega |\nabla u|^p dx}{\displaystyle\int_\Omega |u|^p dx}.
\end{equation}
(We refer the curious reader to \cite{Lq90} and \cite{Peter}.)  Any extremal $u = u(x)$ solves the corresponding Euler-Lagrange Equation
\begin{equation} \label{Euler}
  \mathrm{div}\bigl(|\nabla u|^{p-2}\nabla u\bigr) + \lambda_p |u|^{p-2}u\,=\,0,
  \end{equation}
and
$$v(x,t)\,=\,e^{-\frac{\lambda_p}{p-1}t}u(x)$$
is a solution to Trudinger's Equation.
  In \cite{HLtrud}, the following result is proved: if $v$ is a weak solution of 
$$
\begin{cases}
\displaystyle\frac{\partial \left(|v|^{p-2}v\right)}{\partial t}=\Delta_pv\quad &\text{in}\;\Omega\times (0,\infty)\\
\hspace{.67in}v=g\quad &\text{on}\;\Omega\times\{0\},\\
\hspace{.67in}v=0 \quad &\text{on}\;\partial\Omega\times[0,\infty),
\end{cases}
$$
where $\Delta_p v\,=\,\mathrm{div}(|\nabla v|^{p-2}\nabla v)$,
then the limit 
$$
u=\lim_{t\rightarrow\infty}e^{+\frac{\lambda_p}{p-1}t}v(\cdot,t)
$$
exists in $L^p(\Omega)$ and $u$ is a solution of \eqref{Euler}, possibly identically zero. It is also proved that solutions with extremals as initial data are separable. See \cite{ManVes} for some other related results. We address this question anew  in Theorem \ref{thm:unique}, where we prove a uniqueness result in star-shaped domains. In Corollary \ref{cor:large}, we treat the large time behavior in $C^{1,\alpha}$ domains, not necessarily star-shaped.

\paragraph{Known related results.}
The existence of solutions has been addressed in for instance \cite{Luck}, \cite{Lions} (Paragraph 3.1, Chapter 4) and \cite{Rav}.

To the best of our knowledge, the comparison principles known so far are  limited; in \cite{Luck}, a comparison principle is proved under the extra assumption  
$$\displaystyle\frac{\partial \left(|v|^{p-2}v - |u|^{p-2}u\right)}{\partial t}\in L^1(\Omega_T), $$
where $v$ is a supersolution and $u$ a subsolution. In \cite{Iva}, a comparison principle is proved for a general class of doubly nonlinear equations. The method there is right; however, for the parameter values yielding Trudinger's Equation, we are unable to verify the validity of the proof in that paper.

In \cite{Tuomo}, a Harnack inequality is proved for strictly positive solutions (Theorem 2.1), and in a subsequent comment  it is stated that it is valid also for merely non-negative solutions. A similar result for positive solutions has also been obtained in \cite{GiVe}. 
Local regularity and regularity up to the boundary has been studied in \cite{PorVer}. The equation and its regularity have also been treated in \cite{TuomoHolder}, \cite{Tuomo2}, \cite{Silj} \cite{Sturm1}, \cite{Sturm2} and \cite{Ves}. See also \cite{BhaMar} for a viscosity approach to the equation.

\paragraph{Plan of the paper.}
\small 
In Section \ref{sec:prel}, we introduce some notation and define weak solutions.  In Section \ref{sec:comp}, we prove the comparison principle. In Section \ref{sec:star}, we establish uniqueness of solutions in star-shaped domains with zero lateral boundary condition. In Section \ref{sec:eigen}, we show that one can compare with extremals of the Poincar\'e inequality in $C^{1,\alpha}$ domains, and we use this to study the large time behavior.  Finally, in Section \ref{sec:visc}, we prove that weak solutions are also viscosity solutions when $p\geq 2$. In the Appendix, we give a proof of the fact that the maximum of a subsolution and a constant is again a subsolution. 
\normalsize

\section{Preliminaries}\label{sec:prel}
 We  use standard notation. If  $\Omega$ is a bounded and open set in $\R^n$, we use the notation
$$
\Omega_T = \Omega\times (0,T),
$$ 
and the parabolic boundary  of $\Omega_T$ is
$$\dd_p \Omega_T\,=\, \Omega \times \{0\}\, \cup \,\dd \Omega \times [0,T].$$
 We denote the time derivative of a function $v$ by $v_t$. The positive part of a real quantity $a$ is $a^+ = \max\{a,0\}$.  Let us define the weak supersolutions, subsolutions and solutions of the equation
$$
\left(|v|^{p-2}v\right)_t\,=\,\lap_p v,\quad 1<p<\infty.
$$
Notice that although the functions are non-negative, we shall often write $|v|^{p-2}v$ in place of $v^{p-1},$ the reason being that many auxiliary identities are valid also for sign-changing solutions. 
\begin{de}
  We say that $v\in
  L_{\textup{loc}}^p(0,T;W_{\textup{loc}}^{1,p}(\Omega))$ is a \emph{weak supersolution}
 in $\Omega_T$  if
 \begin{equation}
\label{eq:supersol}
\iint_{\Omega_T}\left(- |v|^{p-2}v\,\phi_t + |\nabla v|^{p-2}\nabla v\cdot \nabla \phi\right) \, dx dt\,\geq\,0
\end{equation}
 holds for any $\phi\in C_0^\infty(\Omega_T),\,\,\phi \geq 0$. Similarly, $u\in
 L_{\textup{loc}}^p(0,T;W_{\textup{loc}}^{1,p}(\Omega))$ is a \emph{weak subsolution} if
 \begin{equation}
   \label{eq:subsol}
 \iint_{\Omega_T}\left(- |u|^{p-2}u\,\phi_t + |\nabla u|^{p-2}\nabla u\cdot \nabla \phi\right) \, dx dt\,\leq\,0.
\end{equation}   
 A function is a \emph{weak solution} if it is both a weak  super- and  subsolution.
\end{de}

\medskip 
By regularity theory, the weak solutions are locally H\"older continuous, see \cite{Tuomo} and \cite{Ves}. It is likely that the weak super- and subsolutions are semicontinuous (upon a change in a set of Lebesgue measure zero),  but we do not know of any reference. For technical reasons, we shall assume that weak \emph{sub}solutions are continuous, in the range $1 < p < 2$, when it is not sure that $\mathrm{ess\,inf} {u} > 0.$
An expedient tool is the weak Harnack inequality in \cite{Tuomo} for weak supersolutions $v\geq 0$ in $\Omega_T$. It implies that if
$$\underset{B\times (t_1,t_2)}{{\mathrm{ess\,inf}}} v \,=\,0$$
  over a strict subdomain, then $v\equiv0$ in $\Omega \times (0,t_1]$.

\section{The comparison principle}\label{sec:comp}

The uniqueness of sufficiently \emph{smooth} solutions that coincide on the parabolic boundary is evident. In particular the time derivative is crucial. Indeed, for two such solutions $u_1$ and $u_2$ in $C(\overline{\Omega}_T)$ with the same boundary and initial values on $\dd_p\Omega_T$, we \emph{formally} use the test function $\phi = H_{\delta}((u_2-u_1)^+),$ where  
 \begin{equation}\label{heavi}
H_\delta(s):=\begin{cases} 1, & s\geq\delta\\
\frac{s}{\delta}, & 0<s<\delta\\
0, & s\leq 0,
\end{cases}
\end{equation}
approximates the Heaviside function, to obtain
\begin{align*}
  &\int_{t_1}^{t_2}\!\!\int_{\Omega}\dd_t\bigl(|u_2|^{p-2}u_2-|u_1|^{p-2}u_1\bigr)\,H_{\delta}((u_2-u_1)^+) \,dxdt\\
  = -\frac{1}{\delta} & \int_{t_1}^{t_2}\!\!\int_{\Omega\cap \{0<u_2-u_1<\delta\}}\langle |\nabla u_2|^{p-2}\nabla u_2 - 
  |\nabla u_1|^{p-2}\nabla u_1,\nabla (u_2 - u_1)^+ \rangle dxdt\,\leq\,0.
\end{align*}
As $\delta \to 0$ it follows that
$$  \int_{t_1}^{t_2}\!\!\int_{\Omega}\dd_t\bigl([|u_2|^{p-2}u_2-|u_1|^{p-2}u_1]^+\bigr)\,dxdt\,\leq\,0.$$
Integrating we see that
$$\int_{\Omega}\,[|u_2|^{p-2}u_2-|u_1|^{p-2}u_1]^+\,dx \Big\vert_{t_2}  \,\leq\, \int_{\Omega}\,[|u_2|^{p-2}u_2-|u_1|^{p-2}u_1]^+\,dx \Big\vert_{t_1}.$$
The last integral becomes zero at $t_1 =0.$
We conclude that $u_2\leq u_1$ and switching the functions we get the desired uniqueness $u_2 =u_1$. This simple proof was purely formal. It is not clear how to rescue this reasoning without access to the time derivatives. Therefore we must pay careful  attention to the proper regularizations\footnote{A similar remark can be made for ``proofs'' of the inequality
  $$\frac{d}{dt}\int_\Omega\!|\nabla u(x,t)|^pdx \,\leq\,0,\qquad u \in L^p(0,T,W^{1,p}_0(\Omega)).$$}.

We have extracted some parts of our proof below of the Comparison Principle from \cite{Iva}. Unfortunately, sign-changing solutions are not susceptible of our treatment. In the next Theorem, the majorant can become infinite, if the parameter $\beta = 1$.

\begin{thm}   \label{thm:comp}  Let  $v\geq 0$ be  a weak supersolution and $u\geq 0$  a weak subsolution in $\Omega_T=\Omega\times (0,T)$, satisfying
\begin{equation}\label{limlim}
\liminf_{(y,s)\to (x,t)} v(y,s)\geq \limsup_{(y,s)\to (x,t)} u(y,s)\quad\text{when}\quad (x,t) \in \partial \Omega \times (0,T) 
\end{equation}
on the lateral boundary.
Then the inequality 
\begin{equation}
\label{eq:intineq}
\int_\Omega \left(|u|^{p-2}u-|\beta v|^{p-2}\beta v\right)^+ dx\, \Big|_{t=t_2}\,\leq\, \int_\Omega \left(|u|^{p-2}u-|\beta v|^{p-2}\beta v\right)^+ dx\, \Big|_{t=t_1}
\end{equation}
holds  for each constant $\beta > 1$ and for a.e. every time $t_1$ and $t_2$, 
  $0<t_1<t_2<T$,  under the following assumptions:
\begin{itemize}
\item[] In the case $p>2$,  $$\underset{\Omega_T}{\mathrm{ess\,inf}}{v}\,>\,0\quad\text{or}\quad\underset{\Omega_T}{\mathrm{ess\,inf}}{u}\,>\,0.$$
\item[] In the case $1 < p < 2$,$$\underset{\Omega_T}{\mathrm{ess\,inf}}{u}\,>\,0, $$
or $u$ is continuous and $\underset{\Omega_T}{\mathrm{ess\,inf}}{v}\,>\,0.$
\end{itemize}
\end{thm}

\begin{proof} Let us first assume $v \geq c$ (and $u$ continuous in the case $1<p<2$). Let $\gamma > 0$. Then $u<v+\gamma$ close enough to $\partial \Omega \times (0,T)$. Since $v\geq c>0$ this implies that 
$$u<v+\gamma =v(1+\gamma/v)\leq v(1+\gamma/c).$$
Define
$$
\tilde v(x,t)\,=\, \left(1+\frac{\gamma}{c}\right) v(x,t),\quad \beta = 1+\frac{\gamma}{c}.
$$
We will prove (\ref{eq:intineq}) with $\beta v$ replaced by $\tilde v$. Adding up \eqref{eq:supersol} for $\tilde v$ and \eqref{eq:subsol} yields
\begin{equation}
\label{eq:firstineq}
\ont \left(|\tilde v|^{p-2}\tilde v-|u|^{p-2}u\right)\phi_t \, dx dt \leq \ont \left(|\nabla \tilde v|^{p-2}\nabla \tilde v-|\nabla u|^{p-2}\nabla u\right)\cdot \nabla \phi \, dx dt,
\end{equation}
for any non-negative  $\phi\in C_0^\infty(\Omega_T)$.

We need a regularization. The Steklov average of  a function $f(x,t)$ is 
$$
f_h(x,t):=\frac{1}{h}\int_{t-h}^t f(x,\tau)\, d\tau, \quad 0<t-h,\,t<T,\, h>0.
$$
See Lemma 3.2 in Chapter 3-(i) in \cite{Ben}.
We note that if
$$
\ont f(x,t)\phi_t(x,t)\, dx dt\,\leq \,0,\quad  \phi\in C_0^\infty(\Omega_T),
$$
then
$$
\ont f(x,t+\tau)\phi_t(x,t)\, dx dt\,\leq\, 0,\quad -h <\tau < 0,\, \phi\in C_0^\infty(\Omega\times(2h,T-h)).
$$
 Integrating  $\tau$ from $-h$ to $0$, we obtain
$$
\ont f_h(x,t) \phi_t(x,t)\, dx dt\,\leq\,0, \quad \phi\in C_0^\infty(\Omega\times(2h,T-h)).
$$
Taking Steklov averages in \eqref{eq:firstineq} we deduce
\begin{equation*}
\ont \left(|\tilde v|^{p-2}\tilde v-|u|^{p-2}u\right)_h\phi_t \, dx dt\, \leq \,\ont \left(|\nabla\tilde v|^{p-2}\nabla\tilde  v-|\nabla u|^{p-2}\nabla u\right)_h\cdot \nabla \phi \, dx dt,
\end{equation*}
for $\phi \geq 0$ belonging to $C_0^\infty(\Omega\times(2h,T-h))$. Since the Steklov average is differentiable with respect to $t$, we may integrate by parts to obtain 
\begin{equation}
\label{eq:3rdineq}
-\ont \dd_t \left(|\tilde v|^{p-2}\tilde v-|u|^{p-2}u\right)_h\phi \, dx dt \leq \ont \left(|\nabla\tilde  v|^{p-2}\nabla\tilde v-|\nabla u|^{p-2}\nabla u\right)_h\cdot \nabla \phi \, dx dt.
\end{equation}
Recall the function $H_{\delta}$ from equation (\ref{heavi}) and define the  function
$$
G_\delta(s)=\begin{cases} s-\frac{\delta}{2}, & s\geq\delta\\
\frac{s^2}{2\delta}, & 0<s<\delta\\
0, & s\leq 0.
\end{cases}
$$
Note that $G'_\delta(s)=H_\delta(s)$ and that $H_\delta(s)$ is an approximation of the Heaviside function $H(s)$. For $2h<t_1<t_2<T-h$, we also define the cut-off function 
$$
\eta_{\varepsilon}(t)=H_{\varepsilon}(t-t_1)-H_{\varepsilon}(t-t_2),
$$
where $\varepsilon > 0$ is small.

Near the boundary $\partial\Omega$ we have $|u|^{p-2}u<|\tilde v|^{p-2}\tilde v$, so that the Steklov average  $[|u|^{p-2}u-|\tilde v|^{p-2}\tilde v]_h < 0$  there. We now choose the  test function \footnote{By estimating $\nabla \phi$ in the same way as $B_h$ on pages \pageref{page:bh}-\pageref{page:bh2}, one can argue that $\phi$ is regular enough to be approximated by smooth functions.} 
$$
\phi(x,t)=\eta_\e(t)H_\delta\left(\left[|u|^{p-2}u-|\tilde v|^{p-2}\tilde v\right]_h\right)
$$
and observe that
$$
\dd_t \,G_\delta\!\left(\left[|u|^{p-2}u-|\tilde v|^{p-2}\tilde v\right]_h\right)=H_\delta\left(\left[|u|^{p-2}u-|\tilde v|^{p-2}\tilde v\right]_h\right)\dd_t\!\left(\left[|u|^{p-2}u-|\tilde v|^{p-2}\tilde v\right]_h\right).
$$
With this particular choice of $\phi$, inequality \eqref{eq:3rdineq} becomes
\begin{align}
\label{eq:5}
&\ont \eta_\e(t) \dd_t G_\delta\!\left(\left[|u|^{p-2}u-|\tilde v|^{p-2}\tilde v\right]_h\right)\, dx dt \\
&\leq \ont\eta_\e(t) \left(|\nabla \tilde v|^{p-2}\nabla \tilde v-|\nabla u|^{p-2}\nabla u\right)_h\cdot \nabla H_\delta\left(\left[|u|^{p-2}u-|\tilde v|^{p-2}\tilde v\right]_h\right) \, dx dt.\nonumber
\end{align}
It is straightforward to send $\e\to 0$. 

We now focus on the left-hand side of \eqref{eq:3rdineq}. As $\e\to 0$ it becomes
\begin{align}\nonumber
&\int_{t_1}^{t_2}\!\!\int_\Omega \dd_t G_\delta\left(\left[|u|^{p-2}u-|\tilde v|^{p-2}\tilde v\right]_h\right)\, dx dt =\\
&\nonumber 
\int_\Omega G_\delta\left([|u|^{p-2}u-|\tilde v|^{p-2}\tilde v]_h\right)\, dx \Big|_{t=t_2}-\int_\Omega G_\delta\left([|u|^{p-2}u-|\tilde v|^{p-2}\tilde v]_h\right)\, dx \Big|_{t=t_1}.
\end{align}
Notice that $G_{\delta}(s)\leq s$.
We may now let $h\to 0$ to obtain\footnote{The limit holds at the Lebesgue points of
  $$t\to \int_{\Omega}\!G_{\delta}(|u|^{p-2}-|\tilde v|^{p-2}\tilde v)dx.$$}
$$
\int_\Omega G_\delta\left(|u|^{p-2}u-|\tilde v|^{p-2}\tilde v\right)\, dx \Big|_{t=t_2}-\int_\Omega G_\delta\left(|u|^{p-2}u-|\tilde v|^{p-2}\tilde v\right)\, dx \Big|_{t=t_1},
$$
for a.e. $t_1$ and $t_2$.
As $h \to 0$, we shall justify that \eqref{eq:5} implies
\begin{align}
&
\int_\Omega G_\delta\left(|u|^{p-2}u-|\tilde v|^{p-2}v\right)\, dx \Big|_{t=t_2}-\int_\Omega G_\delta\left(|u|^{p-2}u-|\tilde v|^{p-2}v\right)\, dx \Big|_{t=t_1}\nonumber\\
& \leq\int_{t_1}^{t_2}\!\!\int_{\Omega} \left(|\nabla \tilde v|^{p-2}\nabla \tilde v-|\nabla u|^{p-2}\nabla u\right)\cdot \nabla  H_\delta\left(|u|^{p-2}u-|\tilde v|^{p-2}\tilde v\right) \, dx dt\label{eq:bla}\\
&=-\frac{1}{\delta}\displaystyle\iint_{\Omega_\delta}\left(|\nabla u|^{p-2}\nabla u-|\nabla \tilde v|^{p-2}\nabla \tilde v\right) \cdot \nabla \left(|u|^{p-2}u-|\tilde v|^{p-2}\tilde v\right) \, dx dt,\nonumber 
\end{align}
where
$$\Omega_\delta = \Omega\times (t_1,t_2)\,\cap\, \{0< |u|^{p-2}u-|\tilde v|^{p-2}\tilde v<\delta\}.$$
The right-hand member in \eqref{eq:5} is (after sending $\e\to 0$)
$$\iint_{\Omega_{\delta,h}} \overbrace{\left(|\nabla \tilde v|^{p-2}\nabla \tilde v-|\nabla u|^{p-2}\nabla u\right)_h}^{A_h(x,t)}\cdot \overbrace{\nabla H_\delta\left(\left[|u|^{p-2}u-|\tilde v|^{p-2}\tilde v\right]_h\right)}^{B_h(x,t)} \, dx dt,$$
where the integration is over the compact subset $\Omega_{\delta,h}$ defined by
$$\Omega_{\delta,h} = \Omega\times (t_1,t_2)\,\cap\, \{0< [|u|^{p-2}u]_h-[|\tilde v|^{p-2}\tilde v]_h<\delta\}.$$
For all small $h$, $\Omega_{\delta,h}\subset  K\times (t_1,t_2)$, for some compact subset $K \subset\subset\Omega$.
We have that if
\begin{align*}
 & A_h \to A_0 \quad\text{strongly in}\quad L^{p/(p-1)}_{loc}(\Omega_T),\\
    &B_h \quad \text{is bounded in} \quad L^p_{loc}(\Omega_T),
\end{align*}\label{page:bh}
then, upon extracting a subsequence, the integrals
converge as they should:
$$\iint_{\Omega_{\delta,h}}\!A_h(x,t)B_h(x,t)\,dxdt \to  \iint_{\Omega_{\delta}}\! A_0(x,t)B_0(x,t)\,dxdt.$$
The convergence of $A_h$ follows from the properties of the Steklov averages. Concerning
the factor $B_h$, which is not a Steklov average, we have that
$$B_h\,=\,\frac{1}{\delta}\nabla \left([|u|^{p-2}u]_h-[|\tilde v|^{p-2}\tilde v]_h\right)$$
in $\Omega_{\delta,h}$ and $B_h$ = 0 otherwise. We have two cases.

If $p > 2$, we may assume that $0\leq u\leq L$, by Corollary \ref{cor:app} in the Appendix. Using H\"{o}lder's inequality and that $[\tilde v ^{p-1}]_h\leq [u^{p-1}]_h$ we have
\begin{align*}
  \big|\left[\tilde v^{p-2}\nabla \tilde v\right]_h\big|\,&\leq\,\left[\tilde v^{p-1}\right]_h^{\frac{p-2}{p-1}}\left[|\nabla \tilde v|^{p-1}\right]_h^{\frac{1}{p-1}}\\&\leq \left[u^{p-1}\right]_h^{\frac{p-2}{p-1}}\left[|\nabla \tilde v|^{p-1}\right]_h^{\frac{1}{p-1}}
    \leq L^{\frac{p-2}{p-1}}\left[|\nabla \tilde v|^{p-1}\right]_h^{\frac{1}{p-1}}.
\end{align*}
Integration yields the bound
\[\begin{split}
\iint_{\Omega_{\delta,h}}\big|\left[\tilde v^{p-2}\nabla  v\right]_h\big|^p\,dxdt\,&\leq\, L^{\frac{p(p-2)}{p-1}}\iint_{\Omega_{\delta,h}}\!\left[|\nabla \tilde v|^{p-1}\right]_h^{\frac{p}{p-1}}\,dxdt,\\
&\leq C \iint_{K\times (t_1,t_2)}|\nabla \tilde v|^p\,dxdt,
\end{split}
\]
when $h$ is small. This bound is  uniform in $h$. A similar estimate holds for the part of $B_h$ that involves $u$. Thus the $L^p$ -norm of $B_h$ is uniformly bounded. This was the case $p >2$. 
  
  For $p < 2$ we proceed as follows. Since $\tilde v> c$ near the lateral boundary, we can replace $u$ by the subsolution $\max\{u,\sigma\}$, where $0<\sigma < c$. See Proposition \ref{prop:maxsubsol} in the Appendix. Then the factors $u^{p-2}$ and $\tilde v^{p-2}$ are bounded from above so that
  $$|[\tilde v^{p-2}\nabla \tilde v]_h|\,\leq\,c^{p-2}[|\nabla \tilde v|]_h,\quad |[u^{p-2}\nabla u]_h|\,\leq\,\sigma^{p-2}[|\nabla u|]_h.$$
  Thus \eqref{eq:5} implies \eqref{eq:bla}, by a standard procedure.
\label{page:bh2}

We recall that for fixed $\gamma>0$ (defining $\tilde v$), $\Omega_\delta$ is compactly contained in $\Omega\times (0,T)$. Using the identity
$$
 \nabla \left(|u|^{p-2}u-|\tilde v|^{p-2}\tilde v\right) =(p-1)|u|^{p-2}\left(\nabla u-\nabla \tilde v\right)+(p-1)\nabla \tilde v\left(|u|^{p-2}-|\tilde v|^{p-2}\right),
$$
and rewriting the right-hand side of \eqref{eq:bla}, we obtain
\begin{align}
&
\int_\Omega G_\delta\left(|u|^{p-2}u-|\tilde v|^{p-2}\tilde v\right)\, dx \Big|_{t=t_2}-\int_\Omega G_\delta\left(|u|^{p-2}u-|\tilde v|^{p-2}\tilde v\right)\, dx \Big|_{t=t_1}\nonumber \\
&\leq-\frac{(p-1)}{\delta}\displaystyle\iint_{\Omega_\delta}|u|^{p-2}\overbrace{\left(|\nabla u|^{p-2}\nabla u-|\nabla \tilde v|^{p-2}\nabla \tilde v\right)\cdot \left(\nabla u-\nabla \tilde v\right)}^{\geq\, 0} dx dt\nonumber \\
&-\frac{(p-1)}{\delta}\displaystyle\iint_{\Omega_\delta}\left(|u|^{p-2}-|\tilde v|^{p-2}\right)\left(|\nabla u|^{p-2}\nabla u-|\nabla \tilde v|^{p-2}\nabla \tilde v\right)\cdot \nabla \tilde v\,  dx dt\label{eq:split} \\
&\leq -\frac{(p-1)}{\delta}\displaystyle\iint_{\Omega_\delta}\left(|u|^{p-2}-|\tilde v|^{p-2}\right)\left(|\nabla u|^{p-2}\nabla u-|\nabla \tilde v|^{p-2}\nabla \tilde v\right)\cdot \nabla \tilde v\,  dx dt.\nonumber 
\end{align}
It is straightforward to see that for a.e. time $t_1$ and $t_2$ we get
$$
\lim_{\delta\to 0}\int_\Omega G_\delta\left(|u|^{p-2}u-|\tilde v|^{p-2}\tilde v\right)\, dx \Big|_{t=t_i}\nonumber =\int_\Omega \left(|u|^{p-2}u-|\tilde v|^{p-2}\tilde v\right)^+\, dx \Big|_{t=t_i}
$$
where  $i =1,2.$

It now remains to verify that  the right-hand side  tends to zero as $\delta\to 0$. The difficulty is that the available inequality $0 < u^{p-1}-\tilde v^{p-1} < \delta$, with a very small $\delta$ does not  necessarily imply that also the quantity $u^{p-2}-\tilde v^{p-2}$ is of the order  $O(\delta)$, if it so happens that  both $u$ and $\tilde v$ are very small. It is at this point that the assumption of a lower bound $c$ is crucial.

The  elementary inequality
$$(1+x)^{\frac{p-2}{p-1}}\,\leq\,1+\tfrac{p-2}{p-1}x,\qquad p >2,\,\,x>0,
$$
implies that
$$u^{p-2} = \left(\tilde v^{p-1}+[u^{p-1}-\tilde v^{p-1}]\right)^{\frac{p-2}{p-1}} < \left(\tilde v^{p-1}+\delta\right)^{\frac{p-2}{p-1}} \leq  \tilde v^{p-2} +\tfrac{p-2}{p-1}\frac{\delta}{\tilde v}.$$
Thus we have arrived at
$$
0<u^{p-2}-\tilde v^{p-2}\leq \tfrac{p-2}{p-1}\,\frac{\delta}{\tilde v}\leq \tfrac{p-2}{p-1}\,\frac{\delta}{c} =  O(\delta) , \quad p>2,
$$
since $\tilde v > v \geq c.$ In the case $1<p<2$, the above elementary inequality is reversed and a similar reasoning yields
$$
0<\tilde v^{p-2}- u^{p-2}\leq \tfrac{2-p}{p-1}\,\frac{\delta}{\tilde v}\leq \tfrac{2-p}{p-1}\,\frac{\delta}{c} =  O(\delta) , \quad p<2.
$$
Hence we can kill the denominator $\delta$ below: 
\begin{align*}
&\Big|\frac{(p-1)}{\delta}\displaystyle\iint_{\Omega_\delta}\left(|u|^{p-2}-|\tilde v|^{p-2}\right)\left(|\nabla u|^{p-2}\nabla u-|\nabla \tilde v|^{p-2}\nabla \tilde v\right)\cdot \nabla \tilde v\,  dx dt\Big| &\\
& \leq C\displaystyle\iint_{\Omega_\delta}\left(|\nabla u|^{p}+|\nabla \tilde v|^{p}\right)\,  dx dt,
\end{align*}
where we have also used Young's inequality for the terms involving $\nabla u$ and $\nabla \tilde v$. Since the integral
$$
\int_{t_1}^{t_2}\!\!\int_{\{|u|^{p-2}u-|\tilde v|^{p-2}\tilde v>0\}}\left( |\nabla u|^p+|\nabla \tilde v|^p\right)\, dx dt
$$
is convergent and $\mathrm{meas}(\Omega_{\delta}) \to 0$ it is clear that 
$$
\displaystyle\iint_{\Omega_\delta}\left(|\nabla u|^{p}+|\nabla \tilde v|^{p}\right)\,  dx dt\,\to\, 0,
$$
as $\delta\to 0$. Therefore, letting $\delta \to 0$ in \eqref{eq:split}, we arrive at
\begin{equation*}
\int_\Omega \left(|u|^{p-2}u-|\tilde v|^{p-2}\tilde v\right)^+ dx \Big|_{t=t_2}\leq \int_\Omega \left(|u|^{p-2}u-|\tilde v|^{p-2}\tilde v\right)^+ dx \Big|_{t=t_1}.
\end{equation*}
This is valid for a.e. $t_1$ and $t_2$. This concludes the case $v\geq c$.

Finally, we note that if we instead assume $u \geq c$ (but  not necessarily $v \geq c$),  then \eqref{limlim} implies that also $v$ is bounded away from zero near $\partial\Omega\times [t_1,t_2]$. By the weak Harnack inequality in \cite{Tuomo}, also $\mathrm{ess\,inf}v > 0,$ the infimum being taken over $\Omega \times [t_1,t_2].$
The same argument goes through again.
\end{proof}

\begin{rem}\label{sobonoll} The conclusion in (\ref{eq:intineq}) is still valid, if the assumption (\ref{limlim}) is replaced by the requirement
  $$(u-v)^+ \in L^p_{loc}(0,T;W^{1,p}_0(\Omega)).$$
  Now the boundary values are taken in Sobolev's sense.
\end{rem}

\medskip

\begin{cor}[Comparison Principle]\label{classic} If inequality (\ref{limlim}) is valid on the whole parabolic boundary $\dd_p \Omega$, then $v\geq u$ in $\Omega_T$.
\end{cor}

\begin{proof} Let $\beta > 1$ be arbitrary. As $t_1 \to 0$, the right-hand side of inequality (\ref{eq:intineq}) approaches zero so that
  $$\int_{\Omega}\bigl(|u|^{p-2}u-|\beta v|^{p-2}\beta v \bigr)^+dx\Big\vert_{t_2}\,\leq 0$$
  for a.e. $t_2$. We conclude that
  $$
  \bigl(|u|^{p-2}u-|\beta v|^{p-2}\beta v \bigr)^+ \,=\,0.
  $$
  The result follows as $\beta \to 1$. 
  \end{proof}

\section{Comparison in star-shaped domains}\label{sec:star}

The restriction $\mathrm{ess\,inf}v > 0$ or $\mathrm{ess\,inf}u > 0$ is \emph{not} assumed in this section.
In  star-shaped  domains, we may prove uniqueness for non-negative solutions, provided that the lateral boundary values are zero. Convex domains are of this type. The initial values, say $g(x)$, have to be attained for a supersolution (or subsolution) $v$ at least in the sense that
\begin{equation}\label{initial}
  \lim_{t \to 0+}\int_{\Omega}\,|v(x,t)^{p-1}-g(x)^{p-1}|\,dx\,=\,0.
  \end{equation}
In the comparison principle below, at least the ``smaller function'' has zero lateral boundary values. (We aim at the eigenvalue problem (\ref{Euler}).) 

\begin{thm}\label{thm:unique} Suppose $\Omega$ is star-shaped
. Let  $v\geq 0$ be  a  weak supersolution with initial values $v(x,0) = g(x)$. Assume that $u\geq 0$ is a  weak subsolution with the same initial values $u(x,0) = g(x)$  and with zero lateral values:
  $$u \in L^p_{loc}(0,T;W^{1,p}_0(\Omega))$$
 In the range $1<p<2$ we assume, in addition, that $u$ is continuous. Then $v\,\geq\,u$ a.e. in $\Omega_T$.
\end{thm}

\medskip

\begin{rem} If the lateral boundary values are taken in the classical sense

  $$\lim_{(y,\tau) \to (x,t)}u(y,\tau)\,=\,0\quad \text{when} \quad (x,t) \in \dd\Omega \times (0,T),$$
the conclusion is still valid.
\end{rem}

\medskip

\begin{proof}  We may exclude the case $v \equiv 0$, since then also $g  \equiv 0$. Thus, $v > 0$ in $\Omega_T$ by the weak Harnack inequality (Theorem 7.1 in \cite{Tuomo}). We may assume that $\Omega$ is star-shaped with respect to the origin.  Consider the weak supersolution
  $$w_{\alpha}(x,t) \,=\,v(\alpha x,\alpha^p t),\qquad 0 < \alpha < 1.$$
  Let $0<t_1<t_2<T.$ Now, again by the same weak Harnack inequality,
  $$w_{\alpha}(x,t)\,\geq\,c_{\alpha} \,>\,0\quad\text{when}\quad (x,t)\in \overline \Omega\times[t_1,t_2],$$
  where $c_{\alpha}$ is a positive constant. (For the original $v$ this inequality may fail.)  Hence, the comparison principle in Theorem \ref{thm:comp}
  applies for $w_{\alpha}$ and $u$ in  the subdomain $\Omega \times [t_1,t_2]$ so that
  $$\int_{\Omega}\,[u^{p-1}-  (\beta w_{\alpha})^{p-1}]^+\,dx\, \Big\vert_{t_2}\,\leq\, \int_{\Omega}\,[u^{p-1}- (\beta w_{\alpha})^{p-1}]^+\,dx\,\Big\vert_{t_1},$$
  when $\beta > 1.$
  Now we can send $t_1$ to $0$.  It follows that
  $$\int_{\Omega}\,[u^{p-1}-  (\beta w_{\alpha}^{p-1})]^+\,dx \,\Big\vert_{t_2}\,\leq\, \int_{\Omega}[g(x)^{p-1}- (\beta g(\alpha x))^{p-1}]^+\,dx.$$
   We can now safely send $\alpha \to 1-$ and $\beta \to 1+$, whence 
  $$\int_{\Omega}\,[u(x,t_2)^{p-1}-  v(x,t_2)^{p-1}]^+\,dx\, \leq\,0.$$
  We conclude that  $u(x,t_2)\leq v(x,t_2)$ when $x \in \Omega.$ It follows that $u\leq v$ a.e. in $\Omega_T$.
  \end{proof}
  
In the next corollary we again assume that the lateral boundary values are taken in the sense that
$$v \in L^p_{loc}(0,T;W^{1,p}_0(\Omega)$$ 
and the initial values are as in (\ref{initial}).

\begin{cor}\label{cor:star} Assume $g\in L^p(\Omega)$ and $\Omega$ star-shaped. Then the solution $v$ of  the following problem
$$
\begin{cases}
\dd_t\left(|v|^{p-2}v\right)=\lap_p v, & (x,t)\in \Omega\times (0,T),\\
v(x,0)=g(x)\geq 0, & x\in \Omega, \\
v(x,t)= 0, &  (x,t)\in \dd\Omega\times (0,T),
\end{cases}
$$
is unique.
\end{cor}
In particular, we obtain that solutions with extremals of \eqref{rayleigh} as initial data are separable.
\begin{cor} Assume that $\Omega$ is a star-shaped domain. Suppose $v$ is a weak solution of
$$
\begin{cases}
\dd_t(|v|^{p-2}v)=\lap_p v,& \Omega\times (0,T),\\
v(x,0)=u_p(x), & x\in  \Omega,\\
v(x,t)=0, & (x,t)\in \dd\Omega\times (0,T),
\end{cases}
$$
where $u_p$ is an extremal of \eqref{rayleigh}. Then 
$$
v= e^{-\frac{\lambda_p t}{p-1}}u_p.
$$
\end{cor}

\section{
  Extremals and large time behavior}\label{sec:eigen} 
When the initial data are comparable to an extremal of \eqref{rayleigh}, we are able to extend the  comparison principle to  more general domains than star-shaped ones. If the boundary is $C^{1,\alpha}$-regular, it is known that the extremals are $C^{1,\alpha}$ up to the boundary and that the estimates are uniform if the $C^{1,\alpha}$-norm of the boundary is uniformly controlled. See Theorem 1 in \cite{Lieb86}. As a consequence we have the following lemma for the extremals of \eqref{rayleigh} under   an exhaustion 
$$\Omega\,=\,\bigcup_{j=1}^\infty \Omega_j,\qquad \Omega_j\subset \subset \Omega_{j+1}$$
of the domain.

\begin{lem}\label{lem:cj} Suppose $\Omega$ is a $C^{1,\alpha}$ domain. Then there are a sequence of  uniformly $C^{1,\alpha}$-regular sets $\Omega_j\subset\subset \Omega $ exhausting $\Omega$ and a sequence of numbers $c_j\to 1$ such that 
$$
c_ju_p^j\leq u_p\quad  \text{ in }\quad \Omega_j.
$$
Here $u_p^j$ and $u_p$ are the extremals of \eqref{rayleigh} in $\Omega_j$ and $\Omega$, respectively.
\end{lem}
\begin{proof} We argue by contradiction. Take $\e>0$. If the result is not true, then for any choice of $\Omega_j$ and for each $j>0$, there is a point $x_j\in \Omega_j$ such that
\begin{equation}\label{erik}
\frac{u_p^j(x_j)}{u_p(x_j)}\geq 1+\e.
\end{equation}
It is clear that $x_j$ must converge to a point in $\dd \Omega$ since the quotient is uniformly convergent to $1$ in every compact subdomain of $\Omega$.

Suppose for simplicity that $x_j\to 0\in \dd\Omega$
and that near the origin the boundary of $\Omega$ is the hyperplane $x_1=0$\footnote{The geometry can be transformed into this upon making a local coordinate transformation with a $C^{1,\alpha}$-function.}. We choose the boundary of $\Omega_j$ to be the hyperplane $x_1=1/j$ . By Theorem 1 in \cite{Lieb86}, the functions $u_p^j$ are uniformly $C^{1,\alpha}(U\cap \{x_1\geq 1/j\})$ for some neighborhood $U$ of the origin. By inspecting  the functions
$$
v^j(x)=u_p^j(x+e_1/j), \quad x\in U\cap \{x_1\geq 0\},
$$
it is easy to conclude from Ascoli's Theorem that up to a subsequence, $\nabla v^j(0)=\nabla u_p^j(e_1/j)$ converges to $\nabla u_p(0)$ when $j \to \infty.$

We may also assume that $x_j = e_1/j + z_j e_1$ where $z_j>0$ and $z_j\to 0$ (otherwise we may just translate the origin along the hyperplane). We note that $\nabla u_p (0)$ and $\nabla u_p^j(x_j)$ both point in the $e_1$ direction. Taylor expansion gives
\begin{equation}\label{peter}
u_p(x_j)=(1/j+z_j)|\nabla u_p(0)|+\O(x_j^{1+\alpha}).
\end{equation}
Since $\nabla u_p^j(e_1/j)$ converges to $\nabla u_p(0)$ we may write
$$
\nabla u_p^j(e_1/j)=|\nabla u_p(0)|e_1+\delta_je_1, \quad \delta_j\to 0.
$$
Again, using Taylor expansion
\begin{align*}
u_p^j(x_j)&=\nabla u_p^j(e_1/j)\cdot e_1 z_j+\O(z_j^{1+\alpha})\\&
=(|\nabla u_p(0)|e_1+\delta_j e_1)e_1z_j+\O(z_j^{1+\alpha})\\
&= |\nabla u_p(0)|z_j+\delta_j  z_j +\O(z_j^{1+\alpha})\\
& = |\nabla u_p(0)|(1/j+z_j)+\O(x_j^{1+\alpha})-|\nabla u_p(0)|/j-\O(x_j^{1+\alpha})+\delta_j z_j+\O(z_j^{1+\alpha}),
\end{align*}
where $\delta_j\to 0$. Hence 
$$
u_p^j(x_j)=u_p(x_j)-|\nabla u_p(0)|/j-\O(x_j^{1+\alpha})+\delta_j z_j+\O(z_j^{1+\alpha})\leq u_p(x_j)+\O(x_j^{1+\alpha})+\delta_j  x_je_1,
$$
since  $|z_j|\leq |x_j|$.
Therefore, using (\ref{peter}) again,
\begin{align*}
\frac{u_p^j(x_j)}{u_p(x_j)}&\leq \frac{u_p(x_j)+\O(x_j^{1+\alpha})+\delta_j x_je_1}{u_p(x_j)} \\
&= 1+\frac{\O(x_j^{1+\alpha})+\delta_j x_je_1}{u_p(x_j)}\\
&= 1+\frac{\O(x_j^{1+\alpha})+\delta_j x_je_1}{x_je_1|\nabla u_p(0)|+ \O(x_j^{1+\alpha})}\,\, \to \,1. 
\end{align*}
This contradicts the antithesis  (\ref{erik}).
\end{proof}

In the next proposition the assumption that $g\geq u_p$ admits multiplication of $u_p$ by arbitrarily small constants. Thus the restriction is crucial only near the boundary $\dd \Omega$.

\begin{prop}\label{prop:eigcomp} Assume that $\Omega$ is a $C^{1,\alpha}$ domain. Suppose $v$ is a non-negative weak solution of
$$
\begin{cases}
\dd_t(|v|^{p-2}v)=\lap_p v,& \Omega\times (0,T),\\
v(x,0)=g(x)\geq 0, & x\in  \Omega,\\
v(x,t)=0, & (x,t)\in \dd\Omega\times (0,T).
\end{cases}
$$
Assume in addition that $g\geq u_p$ where $u_p$ is an extremal of \eqref{rayleigh}. Then 
$$
v\geq e^{-\frac{\lambda_p }{p-1}t}u_p,
$$
in $\Omega\times (0,T)$. Similarly, if $ g\leq u_p$ then the reverse inequality holds.
\end{prop}
\begin{proof} By Harnack's inequality (Theorem 2.1 in \cite{Tuomo}) we can again conclude that $v>0$ in $\Omega_T$. Indeed, if $v$ vanishes at some point then $g$ vanishes identically, which is excluded by the hypothesis.

Let $\Omega_j$ be a sequence of smooth subdomains exhausting $\Omega$. Let $u^j_p$ be the extremal in $\Omega_j$ with eigenvalue $\lambda_p^j$ and the same $L^p$-norm as $u_p$. On page 189 in \cite{Peter}, it is argued for that $\lambda_p^j\to \lambda_p$. By Lemma \ref{lem:cj}, we may choose $\Omega_j$ for which there are constants  $c_j\to 1$, such that $c_ju_p^j\leq u_p$ in $\Omega_j$, for $j$ large enough. Since $\min\{v\}>0$ in $\overline \Omega_j\times (0,T)$ and $v(x,0)=g(x)\geq u_p(x)\geq c_ju_p^j(x)$ for all $x\in \Omega_j$, comparison (Corollary \ref{classic}) implies
$$
v\geq c_je^{\frac{-\lambda^j_p }{p-1}t}u^j_p.
$$
We may pass to the limit in the above inequality and conclude that $v\geq e^{-\frac{\lambda_p }{p-1}t}u_p$. The reversed inequality can be proved similarly.
\end{proof}
As a corollary we obtain that solutions with extremals of \eqref{rayleigh} as initial data are separable.
\begin{cor} Assume that $\Omega$ is a $C^{1,\alpha}$ domain. Suppose $v$ is a weak solution of
$$
\begin{cases}
\dd_t(|v|^{p-2}v)=\lap_p v,\quad \Omega\times (0,T),\\
v(x,0)=u_p(x), \quad x\in  \Omega,\\
v(x,t)=0, \quad x\in \dd\Omega,
\end{cases}
$$
where $u_p$ is an extremal of \eqref{rayleigh}. Then 
$$
v= e^{-\frac{\lambda_p }{p-1}t}u_p.
$$
\end{cor}

We can now obtain pointwise control of the large time behavior assuming that the initial data $g$ satisfy $0\leq g\leq u_p$ where $u_p$ is an extremal of \eqref{rayleigh}.

\begin{cor} \label{cor:large} Assume that $\Omega$ is a $C^{1,\alpha}$ domain and that $g\in  C(\overline \Omega)$ satisfies $0\leq g\leq u_p$ where $u_p$ is an extremal of \eqref{rayleigh}. Then the convergence
  $$u(x) \,=\, \lim_{t \to \infty}  e^{\frac{\lambda_p}{p-1}t}v(x,t) $$
  is uniform in $\overline\Omega$.

If in addition $g\geq w_p$ where $w_p$ is another extremal, then the limit function $u(x)$ is non-zero and therefore an extremal.
\end{cor}
\begin{proof} 
Let $\tau_k$ be an increasing sequence of positive numbers such that $\tau_k\to \infty$ as $k\rightarrow\infty$. In Theorem 1 in \cite{HLtrud} it is proved that $\lim_{k\rightarrow \infty}e^{\frac{\lambda_p \tau_k}{p-1}}v(\cdot,\tau_k)$ exists in $L^p(\Omega)$.  We now argue that this convergence is uniform in $\overline{\Omega}$. 
Let 
$$
v^k(x,t)=e^{\frac{\lambda_p \tau_k}{p-1}}v(x,t+\tau_k).
$$
We remark that $e^{-\frac{\lambda_p t}{p-1}}u_p $ is a solution of Trudinger's equation. By the comparison with extremals (Proposition \ref{prop:eigcomp}) and the fact that $v$ is non-negative, 
\begin{equation*}
0\leq v^k(x,t)\leq e^{-\frac{\lambda_p t}{p-1}}u_p(x)\leq u_p(x),
\end{equation*}
for $(x,t)\in \Omega\times [-1,1]$ for all $k\in \N$ large enough. These bounds together with the local H\"older continuity (Theorem 2.8 in \cite{TuomoHolder} and Theorem 2.5 in \cite{Tuomo2})  give that $v^k$ is uniformly bounded in $C^\alpha(B\times [0,1])$ for any ball $B\subset\subset \Omega$.  Using a covering argument and that $u_p$ is continuous up to the boundary, it is standard to conclude that $v^k$ is equicontinuous in $\overline \Omega\times [0,1]$. From this the result follows as in the proof of Theorem 1.3 in \cite{HL}.

The last part of the statement follows from the observation that in this case we also have the bound 
$$
v^k(x,t)\geq e^{-\frac{\lambda_p t}{p-1}}w_p(x),
$$
which forces the limit to be non-zero. The result then follows from Theorem 1 in \cite{HLtrud}.
\end{proof}

\begin{rem} We note that Corollary \ref{cor:large} can be proved using Theorem \ref{thm:unique} if $\Omega$ is a star-shaped Lipschitz domain and not necessarily a $C^{1,\alpha}$ domain.
\end{rem}

\section{Weak and viscosity solutions}\label{sec:visc}
We shall  prove that weak solutions are also viscosity solutions when $p\geq 2$; see the definition in \cite{BhaMar}. For the general theory we refer to \cite{useless}. As always, this requires at least a comparison principle for classical solutions. 

\begin{thm}\label{visco}  Let $p\geq 2$. A weak solution   $v\geq 0$  of 
$$
\partial_t(v^{p-1})=\lap_p v
$$
in $\Omega_T$  is also a viscosity solution.
\end{thm}
\begin{proof}
We   treat supersolutions and subsolutions separately. The result follows by combining the parts.\\

{\bf \noindent Part 1: supersolution.} By Theorem 7.1 in \cite{Tuomo}, either $v>0$ in $\Omega\times (0,T)$ or there is a  time $t^*$ such that  $v$ is identically zero in $\Omega \times (0,t^*]$ and $v > 0$ in $\Omega \times (t^*,T).$ Therefore, we may as well assume that $v>0$. Suppose that there is a test function\footnote{$C^2$ in $x$ and $C^1$ in $t$} $\phi$ such that $\phi(x_0,t_0)=v(x_0,t_0)$ and $\phi<v$ near $(x_0,t_0)$.\footnote{The definition in \cite{BhaMar} does not require strict inequality away from $(x_0,t_0)$. This however can be accomplished by subtracting $(x-x_0)^4+(t-t_0)^4$ from $\phi$.} We assume, towards a contradiction, that the viscosity inequality for $\phi$ fails:
$$
\partial_t (\phi(x_0,t_0)^{p-1})\,<\,\lap_p \phi(x_0,t_0).
$$
By continuity, the inequalities
$$
\partial_t (\phi^{p-1})<\lap_p \phi,\quad \phi>0,
$$
hold in a  neighborhood of $(x_0,t_0)$. Again, by continuity, we can find a $\lambda>1$ and a possibly smaller cylindrical neighborhood $B_\delta (x_0)\times (t_0-\delta,t_0+\delta)$ such that $\lambda\phi <v$ on the parabolic boundary of $B_\delta (x_0)\times (t_0-\delta,t_0+\delta)$. Note also that the inequality
$$
\partial_t ((\lambda\phi)^{p-1})<\lap_p (\lambda \phi)
$$
remains true in $B_\delta (x_0)\times (t_0-\delta,t_0+\delta)$ by homogeneity. Therefore, Theorem  \ref{thm:comp} implies $\lambda\phi\leq v$ in $B_\delta (x_0)\times (t_0-\delta,t_0+\delta)$. This contradicts that $\phi(x_0,t_0)=v(x_0,t_0)$ since $\phi(x_0,t_0) = v(x_0,t_0)>0$.\\

{\bf \noindent Part 2: subsolution.} Now we prove that $v$ is a viscosity subsolution. Suppose $\phi(x_0,t_0)=v(x_0,t_0)$ and $\phi>v$ otherwise. If $v(x_0,t_0)>0$, we may argue as in the case of a supersolution. Suppose instead that $v(x_0,t_0)=0$. Then $\phi$ is a non-negative function attaining a minimum at $(x_0,t_0)$. Hence, $\phi_t(x_0,t_0)= 0$,\, $\nabla \phi (x_0,t_0)=0$ and $D^2\phi(x_0,t_0)\geq 0$. Hence, 
$$
\partial_t (\phi^{p-1})-\lap_p \phi \leq 0, 
$$
as required.
\end{proof}

\section{Appendix}
 The pointwise maximum of a subsolution and a positive constant is again a subsolution. To be on the safe side, we present a proof that does \emph{not} use the comparison principle.
\begin{prop}\label{prop:maxsubsol}
Let $u \geq 0$ be a weak subsolution in $\Omega_T= \Omega\times (0,T)$ and $c>0$. If $p<2$, assume in addition that $u$ is continuous. Then the function $\max \{u,c\}$  is also a weak subsolution in $\Omega_T$.
\end{prop}
\begin{proof} Since $u$ is a subsolution, we have as in \eqref{eq:3rdineq} on page 7
\begin{equation}
\label{eq:subsolh}
-\int_0^T\!\!\int_\Omega \partial_t([u^{p-1}-c^{p-1}]_h) \phi \,dx dt \,\geq\, \int_0^T\!\!\int_\Omega [|\nabla u|^{p-2}\nabla u]_h \cdot\nabla \phi \,dx dt, 
\end{equation}
for $\phi\geq 0$ in $C_0^\infty(\Omega_T)$, provided that the parameter $h$ in the Steklov average is small enough. Let $\psi \geq 0$ be a test function in $C_0^{\infty}(\Omega_T)$. Now 
$$
\phi(x,t) = \psi(x,t) H_\delta([u^{p-1}-c^{p-1}]_h) 
$$
is a valid test function in \eqref{eq:subsolh} for $0<h<h_{\psi}$ say; here $H_\delta$ is as in the proof of Theorem 1. Using the rule  
$$
\partial_t \,G_\delta\!\left([u^{p-1}-c^{p-1}]_h\right)=H_\delta\left([u^{p-1}-c^{p-1}]_h\right)\partial_t\!\left([u^{p-1}-c^{p-1}]_h\right)
$$
and integrating by parts, we obtain
\[
\begin{split}
\int_0^T\!\!\int_\Omega G_\delta([u^{p-1}-c^{p-1}]_h)\psi_t \,dx dt&\geq 
\int_0^T\!\! \int_\Omega H_\delta([u^{p-1}-c^{p-1}]_h)[|\nabla u|^{p-2}\nabla u]_h \cdot\nabla \psi  \,dx dt\\
&+\frac{1}{\delta}\iint_{\{0<[u^{p-1}-c^{p-1}]_h<\delta\}}[|\nabla u|^{p-2}\nabla u]_h\cdot\nabla [u^{p-1}]_h \psi \,dx dt.
\end{split}
\]
Now we wish to pass $h\to 0$. By standard reasoning  with Steklov averages  we obtain
\begin{align}\label{prat}
  \int_0^T\!\!\int_\Omega G_\delta(u^{p-1}-c^{p-1})\psi_t \,dx dt&\geq 
\int_0^T\!\! \int_\Omega H_\delta(u^{p-1}-c^{p-1})|\nabla u|^{p-2}\nabla u \cdot\nabla \psi  \,dx dt\nonumber\\ 
+{{\limsup_{h\to 0}}}\,\,\frac{1}{\delta}\iint_{\{0<[u^{p-1}-c^{p-1}]_h<\delta\}}&
\psi[|\nabla u|^{p-2}\nabla u]_h\cdot\nabla [u^{p-1}]_h  \,dx dt.
\end{align}
We claim that, upon extracting a subsequence,
\begin{align*}
&\underset{h\to 0}{\mathrm{lim}}\,\iint_{\{0<[u^{p-1}-c^{p-1}]_h<\delta\}}\psi\,[|\nabla u|^{p-2}\nabla u]_h\cdot\nabla [u^{p-1}]_h  \,dx dt\\ &=\,
(p-1)\iint_{\{0<u^{p-1}-c^{p-1}<\delta\}}\psi\, u^{p-2}|\nabla u|^{p}\,dx dt\,\geq\,0,
\end{align*}
so that the last integral in (\ref{prat}) can be thrown away (before we send $\delta$ to zero).

This requires some estimates.
First, we note the convergence
$$
 |[|\nabla u|^{p-2}\nabla u]_h| \to |\nabla u|^{p-2}\nabla u\quad \text{in}\quad L_\text{loc}^{\frac{p}{p-1}}(\Omega_T)
 $$
 for the first factor. The second factor is treated separately depending on the sign of $p-2$.  
 In the case $p > 2$ we have by H\"{o}lder's inequality
\[
\begin{split}
|[u^{p-2}\nabla u]_h |&= \Big|\frac{1}{h}\int_{t-h}^h u^{p-2}|\nabla u| d\tau \Big|\\
&\leq \left(\frac{1}{h}\int_{t-h}^h u^{p-1} d\tau\right)^\frac{p-2}{p-1}\left(\frac{1}{h}\int_{t-h}^h |\nabla u|^{p-1} d\tau\right)^\frac{1}{p-1}\\
&\leq \left([u^{p-1}]_h\right)^\frac{p-2}{p-1}\left(\frac{1}{h}\int_{t-h}^h |\nabla u|^{p-1} d\tau\right)^\frac{1}{p-1}\\
&\leq (c^{p-1}+\delta)^\frac{p-2}{p-1}[|\nabla u|^p]_h^\frac{1}{p},
\end{split}
\]
in the set $$\{0<[u^{p-1}-c^{p-1}]_h<\delta\}.$$
Hence, 
$$
\iint_{ \{0<[u^{p-1}-c^{p-1}]_h<\delta\}} \psi|[u^{p-2}\nabla u]_h|^p\, dx dt \leq C\iint_{\Omega_T}\psi[|\nabla u|^p]_h \,dx dt, 
$$
where $C = (c^{p-1}+\delta)^{p(p-2)/(p-1)}.$ The last integral is uniformly bounded in $h$. 

In the case $p<2$ we instead argue as follows. 
The assumed continuity implies that  $u>c/2$ in the support of $\psi$, when $[u^{p-1}-c^{p-1}]_h>0$ and for $h<h(c)$ small enough.

 As a consequence, when $[u^{p-1}-c^{p-1}]_h>0$, we have the estimate	
\[
\begin{split}
|[u^{p-2}\nabla u]_h |=\Big|\frac{1}{h}\int_{t-h}^h u^{p-2}\nabla u d\tau \Big|
\leq \Big|\frac{c}{2}\Big|^{p-2}\frac{1}{h}\int_{t-h}^h |\nabla u| d\tau
\leq \Big|\frac{c}{2}\Big|^{p-2}[|\nabla u|]_h.
\end{split}
\]
Therefore, 
\[
\begin{split}
\displaystyle\iint_{ \{0<[u^{p-1}-c^{p-1}]_h<\delta\}}\psi |[u^{p-2}\nabla u]_h|^p\, dx dt &\leq \Big|\frac{c}{2}\Big|^{p(p-2)} \iint_{\Omega_T}\psi|[\nabla u]_h|^p \,dx dt.
\end{split}
\]
In both cases we can conclude that the integrand in the problematic term is a product of a function converging in $L_\text{loc}^{p/(p-1)}$ and a function with compact support which is bounded in $L_\text{loc}^p$. This is enough to pass $h\to 0$, after having extracted a subsequence.  

Throwing away the positive term and letting $\delta\to 0$, we obtain
\begin{equation}
\label{eq:almost}
\begin{split}
\iint_{\Omega_T\cap \{u>c\}} \left(u^{p-1}-c^{p-1}\right) \psi_t \,dx dt\geq 
\iint_{\Omega_T\cap \{u>c\}} |\nabla u|^{p-2}\nabla u \cdot \nabla \psi \,dx dt.
\end{split}
\end{equation}
Here we have used that 
$H_\delta(u^{p-1}-c^{p-1})\to \chi_{\{u>c\}}$ in $L^q$ for any $q>0$ and in particular in $L^p$. When multiplied with the  $L^\frac{p}{p-1}$-function $|\nabla u|^{p-2}\nabla u$, the product converges. It now remains to note that
\[\begin{split}
\iint_{\Omega_T\cap \{u>c\}} \left(u^{p-1}-c^{p-1}\right) \psi_t \,dx dt&=\int_0^T\!\!\int_\Omega (\max \{u,c\})^{p-1} \psi_t \,dx dt-\int_0^T\!\!\int_\Omega c^{p-1}\psi_t \,dx dt\\
&= \int_0^T\int_\Omega (\max \{u,c\})^{p-1} \psi_t \,dx dt.
\end{split}
\]
Therefore \eqref{eq:almost} is equivalent to
$$
\int_0^T\!\!\int_\Omega (\max \{u,c\})^{p-1} \psi_t \,dx dt\geq \int_0^T\!\!\int_\Omega |\nabla \max \{u,c\}|^{p-2}\nabla \max \{u,c\} \cdot \nabla \psi \,dx dt.
$$
This is the desired inequality for a subsolution.
\end{proof}

\begin{cor}\label{cor:app}
 Under the same assumptions as in Proposition \ref{prop:maxsubsol}, $u\in L^{\infty}_{loc}(\Omega_T).$
 \end{cor}
 
 \begin{proof} By Lemma 5.1 in \cite{Tuomo}, the weak subsolution $\max\{u,1\}$ is locally bounded. So is, of course, then $u$.
\end{proof}

\paragraph{Acknowledgements:} Erik Lindgren was supported by the Swedish Research Council, grant no. 2017-03736. Peter Lindqvist was supported by The Norwegian Research Council, grant no. 250070 (WaNP).

\bigskip
\noindent {\textsf{Erik Lindgren\\  Department of Mathematics\\ Uppsala University\\ Box 480\\
751 06 Uppsala, Sweden}  \\
\textsf{e-mail}: erik.lindgren@math.uu.se\\

\noindent \textsf{Peter Lindqvist\\ Department of
   Mathematical Sciences\\ Norwegian University of Science and
  Technology\\ N--7491, Trondheim, Norway}\\
\textsf{e-mail}: peter.lindqvist@ntnu.no
\end{document}